\documentclass[12pt]{amsart} % add ",oneside" after "12pt" for single-sided version
\usepackage{setspace,amsthm,amsbsy,amsmath,amssymb,amscd,amsfonts,array,mathrsfs,verbatim,enumerate,xypic,enumitem}
\usepackage[all]{xy}
\usepackage[margin=1.0in]{geometry}
\xyoption{arc} 

{\theoremstyle{plain}
   
   \newtheorem{thm}{Theorem}[section]
   \newtheorem{prop}[thm]{Proposition}     
   \newtheorem{lem}[thm]{Lemma}

}
{\theoremstyle{definition}

   \newtheorem{definition}[thm]{Definition}
   \newtheorem{remark}[thm]{Remark}

}
\numberwithin{thm}{section}

\DeclareFontFamily{OT1}{rsfs}{}
\DeclareFontShape{OT1}{rsfs}{n}{it}{<-> rsfs10}{}
\DeclareMathAlphabet{\curly}{OT1}{rsfs}{n}{it}

\newcommand{\ZZ}{\mathbb{Z}} % integers
\newcommand{\GG}{\mathbb{G}} % group
\newcommand{\QQ}{\mathbb{Q}} % rationals
\newcommand{\PP}{\mathbb{P}} % projective space, projectivization

\newcommand{\m}{\mathfrak{m}}         % maximal ideal of a local ring

\renewcommand{\O}{\mathcal O} % structure sheaf

\newcommand{\ov}{\overline}
\newcommand{\into}{\hookrightarrow}
\newcommand{\be}{\begin{eqnarray*}}
\newcommand{\ee}{\end{eqnarray*}}
\newcommand{\bne}[1]{\begin{eqnarray} \label{#1} }
\newcommand{\ene}{\end{eqnarray}}
\newcommand{\xym}{\xymatrix}
\newcommand{\bp}{\begin{pmatrix}}
\newcommand{\ep}{\end{pmatrix}}

\renewcommand{\H}{\operatorname{H}}

\newcommand{\Id}{\operatorname{Id}}
\newcommand{\Spec}{\operatorname{Spec}}

\newcommand{\PGL}{\operatorname{PGL}}
\newcommand{\Chow}{\operatorname{Chow}}
\newcommand{\Hilb}{\operatorname{Hilb}}
\newcommand{\Gr}{\operatorname{Gr}}     % Grassmannian

\newcommand{\ChowQ}{/\hspace{-1.2mm}/_{Ch}}
\newcommand{\HilbQ}{/\hspace{-1.2mm}/_{H}}
\newcommand{\quotient}{/\hspace{-1.2mm}/}

\begin{document}

\author{Noah~Giansiracusa and William~Danny~Gillam}
%\address{Department of Mathematics, Brown University, Providence, RI 02912}
%\email{noahgian@math.brown.edu, wgillam@math.brown.edu}
%\date{\today}
\title{On Kapranov's description of $\ov{M}_{0,n}$ as a Chow quotient}
\maketitle

\begin{abstract} We provide a direct proof, valid in arbitrary characteristic, of the result originally proven by Kapranov over $\mathbb{C}$ that the Hilbert quotient $(\PP^1)^n\HilbQ\PGL_2$ and Chow quotient $(\PP^1)^n\ChowQ\PGL_2$ are isomorphic to $\ov{M}_{0,n}$.  In both cases this is done by explicitly constructing the universal family of orbit closures and then showing that the induced morphism is an isomorphism onto its image.  The proofs of these results in many ways reduce to the case $n=4$; in an appendix we outline a formalism of this phenomenon relating to certain operads.
\end{abstract}

\section{Introduction}

In this paper we revisit \cite[Theorems 1.5.2, 4.1.8]{K1} where Kapranov shows, over $\mathbb{C}$, that $\ov{M}_{0,n}$ is isomorphic to both the Hilbert and Chow quotients of $(\PP^1)^n$ by the diagonal action of $\PGL_2$.  Our main result is a direct proof of these isomorphisms that is valid over an arbitrary algebraically closed field.  The idea is to construct a flat family over $\ov{M}_{0,n}$ such that the fiber over each $(C,p_1,\ldots,p_n)$ is a union of $\PGL_2$-orbit closures.   This induces a morphism to the Hilbert scheme parametrizing 3-dimensional subschemes of $(\PP^1)^n$.  We prove that this is an isomorphism onto its image, the Hilbert quotient.  Composition with the Hilbert-Chow morphism yields a map from $\ov{M}_{0,n}$ to the Chow quotient, and we show that this is also an isomorphism.  Each of these steps is explained in more detail below.

\subsection{Flat family}\label{intsec:Flat family}

The moduli space $\ov{M}_{0,n}$, $n \ge 3$, parameterizes \emph{stable} $n$-pointed rational curves; these are connected, proper curves of arithmetic genus zero with at worst nodal singularities such that the $n$-marked points are smooth and the dualizing sheaf twisted by the marked points is ample.  The genus condition means that such a curve is a union of irreducible components isomorphic to $\PP^1$ with dual graph forming a tree; the ampleness condition means that each component corresponding to a leaf of the tree carries at least two marked points and each component corresponding to a bivalent vertex carries at least one marked point.

Given a stable curve $(C,p_1,\ldots,p_n)\in\ov{M}_{0,n}$ and an irreducible component $D\subseteq C$, one can retract $C$ onto $D$ to obtain a configuration of $n$ possibly coincident points on $D\cong \PP^1$ that we call the \emph{component configuration} for $D$.  By viewing this as a point of $(\PP^1)^n$ one can form its $\PGL_2$-orbit closure.  The union over all irreducible components $D\subseteq C$ of these orbit closures, when given the reduced induced structure, is a subscheme we denote \[\mathcal{Z}(C,p_1,\ldots,p_n) \subseteq (\PP^1)^n.\]

\begin{thm}\label{thm:Zflat} The subschemes $\mathcal{Z}(C,p_1,\ldots,p_n)$ form a flat family $\ov{Z}$ over $\ov{M}_{0,n}$.
\end{thm}

To prove this, we first construct a flat degeneration that allows for an inductive computation of the Hilbert polynomial of the fiber $\ov{Z}|_{(C,p_1,\ldots,p_n)} = \mathcal{Z}(C,p_1,\ldots,p_n)$ when $C$ is smooth.  This yields flatness over $M_{0,n}$.  We then prove a valuative criterion for a flat family to extend to a compactification and show that this criterion is satisfied for the subschemes $\mathcal{Z}(C,p_1,\ldots,p_n)$ over $\ov{M}_{0,n}$.

\subsection{Hilbert quotient}\label{intsec:Hilbert quotient}

For an algebraic group $G$ acting on a projective variety $X$, there is a Zariski-dense $G$-invariant open subset $U\subseteq X$ such that the orbit closures $\ov{Gu}, u\in U$, form a flat family over the quotient $U/G$ inducing an embedding \[U/G \into \Hilb(X).\]  By definition, the Hilbert quotient is the closure of the image: \[X\HilbQ G := \ov{U/G} \subseteq \Hilb(X).\]  This is independent of the choice of $G$-invariant open subset $U$ \cite[Remark 0.1.8]{K1}.

For our situation, in which $\PGL_2$ acts diagonally on $(\PP^1)^n$, a generic locus $U$ is the set of configurations of $n$ distinct points on the line and its quotient is $M_{0,n}$.   Thus \begin{equation}\label{eq:HilbIncl} M_{0,n} \into \Hilb((\PP^1)^n)\end{equation} and the Hilbert quotient $(\PP^1)^n\HilbQ\PGL_2$ is the closure of $M_{0,n}$ in this Hilbert scheme.

The map in (\ref{eq:HilbIncl}) is induced by the restriction of the flat family $\ov{Z}$ from Theorem \ref{thm:Zflat} to $M_{0,n}$.  Indeed, for $(C,p_1,\ldots,p_n)\in M_{0,n}$ there is only one irreducible component, namely $C\cong \PP^1$, so $\mathcal{Z}(C,p_1,\ldots,p_n)$ is simply the orbit closure of $(p_1,\ldots,p_n)\in (\PP^1)^n$.  Thus $\ov{Z}$ induces a morphism extending (\ref{eq:HilbIncl}) to $\ov{M}_{0,n}$.  Now $\ov{M}_{0,n}$ is proper, so its image under this morphism coincides with the closure of $M_{0,n}$ inside this Hilbert scheme.  Therefore, $\ov{Z}$ yields a map with image equal to the Hilbert quotient $(\PP^1)^n\HilbQ\PGL_2 \subseteq \Hilb((\PP^1)^n)$.

\begin{thm}\label{thm:Ziso}
The map $\ov{Z} : \ov{M}_{0,n} \rightarrow (\PP^1)^n\HilbQ\PGL_2$ is an isomorphism over any algebraically closed field, regardless of characteristic.
\end{thm} 

To prove this, we reduce to the case $n=4$ where the isomorphism can be described quite explicitly.  This gives an independent proof of Kapranov's theorem.

\subsection{Reduction to $n=4$}\label{intsec:Reduction to}

The following two results imply that the map extending (\ref{eq:HilbIncl}) is an isomorphism for each $n$ once we know that it is an isomorphism for $n=4$:

\begin{thm}\label{thm:prodstab} For any $4\le m \le n$ the product of the stabilization maps $s_I : \ov{M}_{0,n} \rightarrow \ov{M}_{0,I}$ over all $I$ of cardinality $m$ is an embedding \be \ov{M}_{0,n} & \into & \prod_{|I|=m} \ov{M}_{0,I}. \ee
\end{thm}

\begin{prop}\label{prop:det} There is a commutative diagram
\[\label{diag:det}\xymatrix{\ov{M}_{0,n} \ar[r] \ar[d] & (\PP^1)^n\HilbQ \PGL_2 \ar[d] \\ \prod_{|I|=4} \ov{M}_{0,4} \ar[r] & \prod_{|I|=4} (\PP^1)^4\HilbQ \PGL_2}\] where the top arrow is induced by $\ov{Z}$ and the bottom arrow is the product of the analogous maps for $n=4$.
\end{prop}

To see that these results imply such a reduction, note that if the bottom arrow is an isomorphism then by Theorem \ref{thm:prodstab} the composition $\ov{M}_{0,n} \rightarrow \prod_{I} (\PP^1)^4\HilbQ \PGL_2$ is an isomorphism onto its image, hence the top surjection is an isomorphism.

We prove Theorem \ref{thm:prodstab} inductively by showing that the map in question separates points and tangent vectors due to a compatibility between the stabilization morphisms and the boundary.  For Proposition \ref{prop:det} we note that the $n=4$ Hilbert quotient lives inside \[\PP \H^0((\PP^1)^4,\mathcal{O}(1,1,1,1))\] and we apply the `det' construction of Knudsen-Mumford  \cite{KM} to induce a morphism from $(\PP^1)^n\HilbQ \PGL_2$ to this space of sections.

\subsection{Chow quotient}\label{intsec:Chow quotient}

When $G$ acts on $X$ as in the definition of the Hilbert quotient, flatness of the generic orbit closures implies that all the cycles $\ov{Gu}$, $u\in U$, have the same dimension and homology class.  There is an embedding \[U/G \into \Chow(X)\] into the Chow variety parameterizing effective algebraic cycles on $X$ with this homology class.  The Chow quotient is then defined as the closure of the image: \[X\ChowQ G := \ov{U/G} \subseteq \Chow(X).\]  In particular, there is an embedding \begin{equation}\label{eq:ChowIncl}M_{0,n} \into \Chow((\PP^1)^n)\end{equation} and the Chow quotient $(\PP^1)^n\ChowQ\PGL_2$ is the closure in this Chow variety.

There are several treatments of the Chow variety which, although in characteristic zero essentially coincide, are not in general equivalent.  We rely on the definition of Koll\'ar \cite[\S I]{Ko}.  It is quite general, although it requires the Chow variety to be seminormal.  

The subscheme $\mathcal{Z}(C,p_1,\ldots,p_n)$ constructed in \S\ref{intsec:Flat family} can be viewed as an algebraic cycle: instead of taking the union of all component configuration orbit closures, we view these same orbit closures as effective cycles and take their sum.  Since the union is irredundant, this is the same as the fundamental cycle of the union.  This induces a morphism \[[\ov{Z}] : \ov{M}_{0,n} \rightarrow \Chow((\PP^1)^n)\] extending (\ref{eq:ChowIncl}) and with image $(\PP^1)^n\ChowQ \PGL_2$, as follows.  Since $\ov{M}_{0,n}$ is smooth, so in particular seminormal, the map $\ov{Z} : \ov{M}_{0,n} \rightarrow \Hilb((\PP^1)^n)$ factors through the seminormalization of this Hilbert scheme. By \cite[Theorem 6.3]{Ko} there is a Hilbert-Chow morphism $\Hilb^{sn}(X) \rightarrow \Chow(X)$ for any $X$.  Thus $\ov{Z}$, viewed as a family of cycles, induces the composition $\ov{M}_{0,n} \rightarrow \Hilb^{sn}((\PP^1)^n) \rightarrow \Chow((\PP^1)^n)$.

\begin{thm}\label{thm:ZChiso}
The map $[\ov{Z}] : \ov{M}_{0,n} \rightarrow (\PP^1)^n\ChowQ\PGL_2$ is an isomorphism.
\end{thm} 

This also reduces to the case $n=4$, as explained below.

\subsection{The case $n=4$}\label{intsec:The case}

For $(\PP^1)^4$ each 3-dimensional orbit closure is a hypersurface cut out by a multi-homogeneous form, so the relevant components of the Hilbert scheme and Chow variety coincide: \[\Hilb((\PP^1)^4) = \Chow((\PP^1)^4) = \PP\H^0((\PP^1)^4,\O(1,1,1,1)) = \PP^{15}.\]  

\begin{lem}\label{lem:n=4}
The map $M_{0,4} \into \PP\H^0((\PP^1)^4,\O(1,1,1,1))$ sending a configuration of four distinct points on the line to the form cutting out its orbit closure extends to a map from $\ov{M}_{0,4}$ which is an isomorphism onto its image: \[\ov{M}_{0,4}~\widetilde{\rightarrow}~(\PP^1)^4\HilbQ\PGL_2 = (\PP^1)^4\ChowQ\PGL_2 \into \PP\H^0((\PP^1)^4,\O(1,1,1,1)).\]
\end{lem}

The classical cross-ratio can be used to explicitly describe these forms for generic orbit closures.  It is then evident from their description that they also describe orbit closures for configurations with at most two coincident points.  Since $\ov{M}_{0,4}$ allows at most two points to come together, the proof of this lemma follows in a straightforward manner.

\noindent \emph{Proof of Theorem \ref{thm:ZChiso}}.  Koll\'ar's definition of the Chow variety is covariant with respect to proper maps \cite[Theorem 6.8]{Ko}, so the natural projections $(\PP^1)^n \rightarrow (\PP^1)^4$ induce, via push-forward of cycles, morphisms $\Chow((\PP^1)^n) \rightarrow \Chow((\PP^1)^4)$.  This obviates the need for an analogue of Proposition \ref{prop:det} as it yields the following commutative diagram:
\[\label{chowdiag}\xymatrix{\ov{M}_{0,n} \ar@{->>}[r] \ar[d] & (\PP^1)^n\ChowQ \PGL_2 \ar@{^{(}->}[r] \ar[d] & \Chow((\PP^1)^n) \ar[d] \\ \prod_I \ov{M}_{0,4} \ar@{->>}[r] & \prod_I (\PP^1)^4\ChowQ \PGL_2 \ar@{^{(}->}[r] & \prod_I\Chow((\PP^1)^4)}\] 
The argument that Theorem \ref{thm:ZChiso} holds for all $n$ since it holds for $n=4$ is now exactly the same as in \S\ref{intsec:Hilbert quotient} for the Hilbert quotient. \hfill $\Box$

\begin{remark}
For any theory of Chow varieties which admits proper push-forward morphisms \emph{and} a Hilbert-Chow morphism this argument yields an isomorphism to the Hilbert quotient based on the $n=4$ isomorphism.  Since we rely on Koll\'ar's treatment of the Chow variety this only proves an isomorphism of $\ov{M}_{0,n}$ with its image in the seminormalization of the Hilbert scheme.  This is why we instead rely on Proposition \ref{prop:det}.
\end{remark}

\subsection{Operads}
It is curious that once the family $\ov{Z} \rightarrow \ov{M}_{0,n}$ is constructed, the proof that it induces isomorphisms to both the Hilbert quotient and Chow quotient reduces to the case $n=4$, and our main tool for this reduction, Theorem \ref{thm:prodstab}, is proven using a compatibility with the $n=4$ case.  In an appendix we sketch a categorical framework underlying this phenomenon, using ideas from the theory of cyclic operads introduced by Getzler and Kapranov \cite{GK}.

\subsection{Background/motivation}

Kapranov's original proof is accomplished by constructing an impressive series of isomorphisms.  First, he shows that the Gelfand-MacPherson correspondence extends to Chow quotients \cite[Theorem 2.2.4]{K1}: \[(\PP^1)^n \ChowQ \PGL_2 \cong \Gr(2,n) \ChowQ \GG_m^n.\] Next, he shows that $\Gr(2,n) \ChowQ \GG_m^n$ is isomorphic to the closure in $\Chow(\PP^{n-2})$ of the locus of rational normal curves through $n$ fixed, generic points in $\PP^{n-2}$ \cite[Corollary 4.1.9]{K1}.  Finally, he shows that this space of Veronese curves is isomorphic to $\ov{M}_{0,n}$ \cite[Theorem 0.1]{K2}.  The delayed appearance of $\ov{M}_{0,n}$ is perhaps because only this last isomorphism extends to an isomorphism between the universal families.  Kapranov derives the corresponding result about Hilbert quotients by proving that the relevant Hilbert-Chow morphism is an isomorphism.  Throughout, he relies on results of Barlet valid only over $\mathbb{C}$ to study these various Chow varieties and remarks that the case of positive characteristic is more subtle and not considered in his paper \cite[0.1.5]{K2}.  

Our original motivation for exploring this topic was to understand Kapranov's isomorphisms in as direct a way as possible, and to provide a simple geometric description of the universal family of $\PGL_2$ orbit closures.  The first author has applied ideas from this paper to study a family of Chow quotients that also turn out to be isomorphic to $\ov{M}_{0,n}$ but for which the Gelfand-MacPherson correspondence does not play a role.  See \cite{Gi}.  

\subsection*{Outline} We introduce basic notation and constructions in \S\ref{section:Notation}, then in \S\ref{Proof of prodstab} provide a proof of Theorem \ref{thm:prodstab}.  In \S\ref{section:orbitclosures} we study the structure of orbit closures in $(\PP^1)^n$.  Section \S\ref{section:crossratio} contains the proof of the $n=4$ case (Lemma \ref{lem:n=4}) and \S\ref{section:Hilbertpolynomial} uses ideas from it to compute various Hilbert polynomials.  In \S\ref{section:Proof of Zflat} we use these computations to prove that the family $\ov{Z}$ is flat (Theorem \ref{thm:Zflat}), which by the reduction described above completes the proof of Theorems \ref{thm:Ziso} and \ref{thm:ZChiso} based on Proposition \ref{prop:det}, which is proven in \S\ref{Proof of det}.  We conclude with an appendix on operads as an alternate formalism to prove certain results in this paper.

\subsection*{Acknowledgements} We thank Dan Abramovich, Matt DeLand, Jeffrey Giansiracusa, Joe Ross, and Jonathan Wise for helpful conversations, and the anonymous referee for a careful reading and useful suggestions.  N.G. was partially supported by funds from National Science Foundation (NSF) award DMS-0901278.  W.D.G. was supported by an NSF Postdoctoral Fellowship.

\section{Notation}\label{section:Notation} 

Throughout the paper let $X := \mathbb{P}^1$ and $G := \PGL_2$, so that $G$ acts diagonally on $X^n$.  

\subsection{Indexing sets}\label{section:Indexing sets} Let $K+L$ denote the disjoint union of sets $K$ and $L$, and write $K+1$ for the disjoint union of $K$ with the one-element set $\{*\}$.  For an inclusion of finite sets $I \into N$ let $\pi_I : X^N \to X^I$ be the obvious projection and if $I^c \neq \emptyset$ then let $i_I : X^{I+1} \into X^N$ be the closed embedding such that $i_I(x)_j$ is $x_j$ if $j\in I$ and $x_*$ otherwise.

\subsection{Diagonals} It is convenient to introduce the following notation related to diagonals: \be \Delta_i & := & \{ x \in X^n : x_j = x_k {\rm \; for \; all \; } j,k \ne i \} \\ \Delta_\bullet & := & \Delta_1 \cup \cdots \cup \Delta_n \\ \Delta_{\rm big} & := & \{ x \in X^n : {\rm \; there \; is \; } i\ne j {\rm \; with \;} x_i=x_j \} \\ U_n & := & X^n \setminus \Delta_{\rm big} \\ \Delta_{\rm small} & := & \{ x \in X^n : x_i=x_j {\rm \; for \; all \;} i,j \} \ee  

\subsection{Partition from a configuration} A point $x \in X^n$ determines a partition $P = (P_1, \dots, P_l)$ of $[n] := \{ 1, \dots, n \}$ with some ordering, e.g., lexicographical, that we call the \emph{type} of $x$.  By definition, $i$ and $j$ are in the same part if and only if $x_i = x_j$.  The type of $x$ depends only on its $G$-orbit.  The \emph{generic} type is the one with all distinct coordinates: $l=n$.  The cycle $\overline{Gx} \in Z_*(X^n)$ given by the closure of the orbit of $x$ depends only on the $G$-orbit of $x$ and we show in Proposition~\ref{prop:Chowclass} that its Chow class $[ \overline{Gx} ] \in A_*(X^n)$ depends only on the type of $x$.

A partition $P$ of $[n]$ determines a closed embedding $\Delta_P : X^l \into X^n$ characterized by $\Delta_P(y)_i = y_j \iff i \in P_j$.  The locus of points in $X^n$ of type $P$ is $\Delta_P(U_l)$, and \[\Delta_P := \Delta_P(X^l) = \overline{\Delta_P(U_l)} \subseteq X^n.\]  In particular, if $x$ has type $P$ then $x = \Delta_P(y)$ for $y \in U_l$ and $\overline{Gx} = \Delta_P( \overline{Gy})$ .  This allows us to reduce many questions about $\overline{Gx}$ to the case where $x$ has generic type.

\subsection{Component configurations}\label{section:Component configurations}
Let $k$ be a field and $(C,p_1,\dots,p_n) \in \ov{M}_{0,n}(k)$.  The fact that the $p_i$ are $k$-points and the curve is stable implies that every irreducible component $D\subseteq C$ is isomorphic to $X_k$ and every node is a $k$-point.  There is a unique retract $\pi_D : C \to D$ of the inclusion $D \into C$ and after choosing an isomorphism $D \cong X_k$ we can regard $(\pi_D(p_1),\dots,\pi_D(p_n))$ as a point of $X_k^n$, called the \emph{component configuration} for $D$.  The $G$-orbit closure of this point is independent of the choice of isomorphism $D \cong X_k$.  Stability guarantees that at least three of the points $\pi_D(p_i)$ remain distinct on $D$, so this orbit closure is 3-dimensional.

\section{Proof of Theorem \ref{thm:prodstab}}\label{Proof of prodstab}

Theorem \ref{thm:prodstab} is used in the proof of our two main theorems to allow for a reduction to the case $n=4$ (recall \S\ref{intsec:Hilbert quotient}), but it is potentially useful in other contexts so we include its proof here before launching into the main content of this paper.  First, we recall some facts about the forgetful maps on $\ov{M}_{0,n}$.

\subsection{Stabilization morphisms} \label{section:Stabilization morphisms} For each inclusion $I \into N$ of finite sets of cardinality $ \geq 3$ there is a \emph{stabilization} morphism $s_I : \ov{M}_{0,N} \to \ov{M}_{0,I}$ which forgets the markings not in $I$ then stabilizes.  The morphism $s_{N} : \ov{M}_{0,N+1} \to \ov{M}_{0,N}$ is the universal curve over $\ov{M}_{0,N}$ \cite[\S 2]{Kn}.  It is straightforward to see that the stabilization morphisms are compatible with the boundary inclusions, in the sense that for every partition $N = K + L$ and every $I \subseteq N$ with $|I \cap K|, |I \cap L| \geq 2$, the diagram $$ \xym@C+60pt{ \ov{M}_{0,K+1} 
\times \ov{M}_{0,L+1}  \ar@{^(->}[d] \ar[r]^-{ s_{K \cap I +1} \times s_{L \cap I +1} } & \ov{M}_{0,K \cap I +1} \times \ov{M}_{0,L \cap I+1} \ar@{^(->}[d] \\ \ov{M}_{0,N} \ar[r]^-{s_I} & \ov{M}_{0,I}  } $$ commutes.  The vertical arrows here indicate the inclusion of boundary divisors in these moduli spaces; these maps are given in modular terms by attaching the two curves corresponding to the direct product along their last marked point.  If, say, $I \cap L = \emptyset$, then the diagram $$ \xym@C+60pt{ \ov{M}_{0,K+1} \times \ov{M}_{0,L+1}  \ar@{^(->}[d] \ar[rd]^-{s_I \pi_1} \\ \ov{M}_{0,N} \ar[r]^-{s_I} & \ov{M}_{0,I}  } $$ commutes, and if $|I \cap L| =1$ then the following diagram commutes: $$ \xym@C+60pt{ \ov{M}_{0,K+1} \times \ov{M}_{0,L+1}  \ar@{^(->}[d] \ar[rd]^-{s_{I \cap K+1} \pi_1} \\ \ov{M}_{0,N} \ar[r]^-{s_I} & \ov{M}_{0,I}  } $$

\subsection{Separating points and tangents}\label{section:Separating}

To prove Theorem \ref{thm:prodstab} we first note that the product of stabilization morphisms $\varphi := \prod_{|I|=4}s_I$ to the $\binom{n}{4}$ copies of $\ov{M}_{0,4}$ factors through the product over $I$ of cardinality $m\ge 4$, so it is enough to prove the case $m=4$.  Moreover, since $\varphi$ is a projective morphism between varieties it is enough to show that it separates points and tangent vectors.  We use induction on $n$, noting that the case $n=4$ is trivial.

\subsubsection*{Claim: $\varphi(k) : \ov{M}_{0,n}(k) \rightarrow \prod_I \ov{M}_{0,4}(k)$ is injective for any field $k$}  Let $x \ne y$ be distinct points of $\ov{M}_{0,n}(k)$.  The restriction of $\varphi(k)$ to $M_{0,n}(k)$ is injective, so we can assume $x$ is in a boundary component \[D_{K,L} = \ov{M}_{0,K+1}(k) \times \ov{M}_{0,L+1}(k)\] corresponding to a partition $[n] = K + L$, with $|K|,|L|\ge 2$.  Write $x = (x_1,x_2)$.  For any two-element subsets $K' \subseteq K$, $L' \subseteq L$ with $I = K' + L'$, the map $s_I$ takes $x$ to the element of $\ov{M}_{0,4}(k)$ corresponding to the reducible curve where the node separates the points marked by $K'$ and $L'$; in particular, $s_I$ is constant on $D_{K,L}$.  If $y \notin D_{K,L}$ then by Lemma \ref{lem:nodes}, below, there is $K'\subseteq K, L'\subseteq L$ such that for $I = K' + L'$ we have $s_I(y) \ne s_I(x)$.  Thus we can assume $y = (y_1,y_2) \in D_{K,L}$ with, say, $x_1\neq y_1$.  By the inductive hypothesis there is $I \subseteq K+1$ of size four such that  $s_I  : \ov{M}_{0,K+1} \rightarrow \ov{M}_{0,4}$ satisfies $s_I(x_1)\ne s_I(y_1)$.  On the other hand, by the compatibility of stabilization and boundary described above there is $J \subseteq [n]$ of size four such that $s_I \pi_1 :   \ov{M}_{0,K+1} \times \ov{M}_{0,L+1} \rightarrow \ov{M}_{0,4}$ agrees with the restriction of $s_J$ to $D_{K,L} \subseteq \ov{M}_{0,n}$.  Hence $s_J(x) \ne s_J(y)$.

\subsubsection*{Claim: $d \varphi(x) : T_x \ov{M}_{0,n} \rightarrow \prod_I T_{s_I (x)} \ov{M}_{0,4}$ is injective for each $x \in \ov{M}_{0,n}(k)$} This holds for $x \in M_{0,n}(k)$ so we can assume $x = (x_1,x_2)\in D_{K,L}$, as before.  It follows from the compatibility of stabilization and boundary, together with the inductive hypothesis, that $d \varphi (x)$ is injective when restricted to \[T_x D_{K,L} = T_{x_1} \ov{M}_{0,K+1} \oplus T_{x_2} \ov{M}_{0,L+1}.\]  On the other hand, for any two-element subsets $K' \subseteq K$, $L' \subseteq L$ with $I = K' + L'$ we observed above that $s_I$ is constant on $D_{K,L}$.  Therefore, its differential at $x$ kills $ T_x D_{K,L}$ and hence induces a map $$ d s_I (x) : N_x \to T_{s_I(x)} \ov{M}_{0,4} , $$ where $N$ is the normal bundle of $D_{K,L}$ in $\ov{M}_{0,n}$.  We will show that there is always a choice of $I$ such that this induced map is nonzero.  There is a unique node $q$ on the curve $x$ separating the markings in $K$ from those in $L$.  If $C_K,C_L$ are the components meeting at $q$, named in the obvious way, then $N_x \cong T_q C_K \otimes T_q C_L$.  By stability, $C_K$ has at least two special points $s_1, s_2$ besides $q$.  If $s_i$ is a marking, set $p_i := s_i$, otherwise we can choose a marking $p_i$ separated from $C_K$ by the node $s_i$.  Let $K' := \{ p_1, p_2 \} \subseteq K$.  Let $L' \subseteq L$ be defined analogously and set $I := K' + L'$.  Then the map from $x$ to its stabilization $s_I(x)$ does not contract $C_K$ or $C_L$ and hence $d s_I(x) : T_q C_K \otimes T_q C_L \rightarrow T_{s_I(x)} \ov{M}_{0,4}$ is nonzero.  This completes the proof, except for the lemma alluded to above: 

\begin{lem} \label{lem:nodes} For $(C,p_1,\ldots,p_n) \in \ov{M}_{0,n}$ and a partition $[n]=K + L$ with $|K|,|L|\ge 2$, there is a node of $C$ separating $K$ from $L$ if, and only if, for every two-element subsets $I \subseteq K$, $J \subseteq  L$, there is a node of $C$ separating $I$ from $J$. \end{lem}

\begin{proof} We proceed by induction on the number of irreducible components of $C$.  If $C$ is irreducible, there is nothing to prove, otherwise pick an irreducible component $D$ of $C = D \cup_q C'$ with at least two markings and exactly one node $q$.  There is nothing to prove if $D$ contains markings from both $K$ and $L$, so we can assume all of the markings $K'$ on $D$ are in $K$.  By induction we know the result holds for $C'$ with the partition $(\{q\} + K \setminus K' ) + L$ of the markings on $C'$, and this implies the result we want for $C$.  

We note that this can also be proven without induction (and we thank the referee for the following argument).  For any set $S$ of marked points there is a unique minimal subtree $T(S)$ of irreducible components that contain all marked points in $S$. A node separates $K$ from $L$ if and only if the trees $T(K)$ and $T(L)$ contain no irreducible component $l$ in common. If a node separates $K$ and $L$, the same node separates all two element subsets. If not, because of the tree structure, there exist $p_1,p_2 \in K$ and $q_1,q_2 \in L$ such that $T(p1,p2)$ and $T(q1,q2)$ both contain a given common component $l$; therefore these pairs are not separated.
\end{proof}  

\section{Orbit closures in $(\PP^1)^n$} \label{section:orbitclosures} 

In this section we study the $G$-orbit closure of points in $X^n$.  Throughout, we work over a field $k$, though we write $X^n$ instead of $X^n_k$, and every point is a $k$-point even though we write $x \in X^n$ instead of $x \in X^n(k)$.

\begin{lem}\label{lem:torsor} If $x \in X^n$ has type $P_1,\dots,P_l$ with $l \geq 3$, then $g \mapsto gx$ defines an isomorphism $G \cong Gx$.  In other words, the orbit $Gx$ is a $G$-torsor. \end{lem}

\begin{proof} This follows from the fact that $G$ acts simply transitively on $U_3$. \end{proof} 

\subsection{Chow classes of orbit closures}
For $I \subseteq [n]$, $y \in X^n$ set $W_I(y) := \pi_I^{-1}(\pi_I(y))$ and $Z_I(y) :=  \pi^{-1}_{I^c}(\pi_{I^c}(y))$.  Then $Z_I(y) \cong X^I$ and the class $H_I := [Z_I(y)] \in A_I(X^n)$ is independent of the choice of $y$.  The classes $H_I$ form a linear basis for $A_*(X^n) \cong \ZZ^{2^n}$.  Note that $W_I(y)$ and $Z_I(y)$ intersect transversely in the single point $y$ and the corresponding cycle classes $H_{I^c}$ and $H_I$ are dual under the perfect intersection pairing $A_*(X^n) \times A_{n-*}(X^n) \to \ZZ$.

\begin{prop} \label{prop:Chowclass} Let $x \in X^n$ be a point of type $(P_1,\dots,P_l)$, with $l \geq 3$.  Then $$[\overline{G x}] = \sum H_I  \in A_3(X^n),$$ where the sum is over three-element subsets $I \subseteq [n]$ with $| I \cap P_m | \leq 1$ for all $m$.  In particular, the Chow quotient $X^n\ChowQ G$ is embedded in the Chow variety $\Chow_{\beta} X^n$ parameterizing cycles of class \vspace{-.1in} $$\beta := \sum_{|I|=3} H_I $$ since this is the Chow class of a generic orbit closure.
\end{prop}

In \cite[Proposition 2.1.7]{K1} the homology class of a generic orbit closure for the action of $\PGL_{d+1}$ on $(\mathbb{P}^d)^n$ is given; when $d=1$ and $l=n$ it coincides with the above formula.

\begin{proof} By Lemma \ref{lem:torsor}, $\overline{G x}$ has dimension three.  Since cycles modulo rational equivalence, algebraic equivalence, and numerical equivalence coincide on $X^n$, all we have to show is that the cycle class $[\overline{G x}]$ has the right intersection numbers with cycle classes of codimension three.  Fix any $y \in U_n$.  It is enough to show: (i) $\overline{G  x} \cap W_I(y) = \emptyset$ when $I$ does not satisfy the condition of the proposition, and (ii) the cycles $\overline{G  x}$ and $W_I(y)$ meet transversely in a single point when $I$ satisfies the condition of the proposition.

For (i), if $I = \{ i,j,k \}$ does not satisfy the condition of the proposition, then we can find $m \in \{ 1, \dots, l \}$ with, say, $i,j \in P_m$, after possibly relabelling $i,j,k$.  That is, $x_i = x_j$, hence $z_i = z_j$ for every $z \in \overline{Gx}$, since this condition is closed and $G$-invariant.  But $y \in U_n$, so $y_i \neq y_j$, hence $z_i = y_i \neq y_j = z_j$ for every $z \in W_{ijk}(y)$, so (i) is proved.  

For (ii), if $I = \{ i,j,k \}$ satisfies the condition of the proposition, then $x_i, x_j, x_k \in \PP^1$ are distinct.  Now $y_i, y_j, y_k$ are also distinct, since $y \in U_n$, so there is a unique $g \in G$ with $$gx_i = y_i, \; gx_j = y_j, \; gx_k = y_k.$$  After replacing $x$ by $gx$, we can assume without loss of generality that $g = \Id$.  Then it is clear that $x \in W_I(y) \cap Gx$, so we want to show that $x$ is the only point in $W_I(y) \cap \overline{Gx}$ and that the intersection is transverse at $x$.  Certainly $Gx \cong G$ is smooth of dimension three at $x$, so the transversality amounts to saying that any nonzero tangent vector $v \in \mathfrak{g}$ to the identity moves $x$ out of $W_I(y)$ infinitesimally, i.e., $v \cdot x \in T_x X^n$ should not be in $T_x W_I(y)$.  This is equivalent to saying $v$ will infinitesimally move one of $$x_i = y_i, \, x_j=y_j, \, x_k=y_k,$$ or equivalently that the $G$-action on $\PP^1$ is infinitesimally effective on triples of distinct points.  This is clear since it is a principal action.  Finally, suppose $z \in \overline{Gx} \cap W_I(y)$, so we want to show that $z=x$.  The map from $Gx$ to $(\PP^1)^3$ is a degree 1 map. Considering its extension to $\ov{Gx}$, the preimage of the point $(y_i,y_j,y_k)$ remains connected since the generic fiber is connected. Therefore, the preimage cannot contain any point other than $(x_i,x_j,x_k)$ since this point is isolated and is in the preimage.
\end{proof}

\subsection{Boundary of an orbit closure}

\begin{lem} \label{lem:boundary} Let $x \in X^n$ be a point of type $P=(P_1,\dots,P_l)$, with $l \geq 3$.  The boundary $\overline{Gx}\setminus Gx$ is a union of two-dimensional orbit closures depending only on the type of $x$: \be \overline{Gx} \setminus Gx & = & \bigcup_{i=1}^l \Delta_{P_i, P^c_i} \ee This union is irredundant. \end{lem}

\begin{proof}  By working in $\Delta_P$, it suffices to treat the case $x \in U_n$.  Here we want to show that $\overline{Gx} \setminus Gx = \Delta_\bullet$. Fix $i \in [n]$ and let $y \in \Delta_i$ be the point with $y_i = 0$ and $y_j = \infty$ for $j \neq i$.  Then $\Delta_i = \overline{Gy}$, so to obtain the containment $\supseteq$ it is enough to show that $y \in \overline{Gx}$.   Since $x \in U_n$, we can assume, after possibly replacing $x$ with a different point in its $G$ orbit, that $x_i = 0$ and $x_j = \infty$ for some $j \in [n]$.  In what follows we use the notion of a ``limit'' in a one-parameter subgroup for conceptual clarity, but the formalism of this in terms of DVRs, valid over an arbitrary field, is immediate.  The subgroup of $G$ stabilizing $0,\infty$ is a copy of the multiplicative group $\GG_m$ acting on $\PP^1$ by the usual scaling action.   The limit, as $g \to \infty$ in this $\GG_m$, of any $z \in \PP^1 \setminus \{ 0 \}$ is $\infty \in \PP^1$, so the fact that $x_j \in \PP^1 \setminus \{ 0 \}$ for $j \neq i$ implies that $\lim_{g \to \infty} gx = y$.

It remains to prove the containment $\subseteq$, so suppose $y \in \overline{Gx} \setminus Gx$.  If all the coordinates of $y$ are equal, then $y \in \Delta_{\rm small} \subseteq \Delta_\bullet$, so assume $y_i \neq y_j$ for some $i,j \in [n]$.  Since $y \notin Gx$, there is $k \in [n] \setminus \{ i,j \}$ with, say, $y_j = y_k$.  Now $\Delta_\bullet$ is $G$-invariant, so after replacing $y$ with a different point in its $G$-orbit we can assume that $y_i = 0$ and $y_j = y_k = \infty$.  We can also assume $x_i = 0, x_j = \infty, x_k = 1$.  Our goal is then to show that, in fact, $y_l = \infty$ for every $l \neq i$.  Choose $g_1,g_2,\dots \in G$ such that $g_t x \to y$ as $t \to \infty$ and choose a sequence $w_1, w_2,\ldots \in \PP^1 \setminus \{ 0 , \infty \}$ converging to $\infty$.  Since $x_i, x_j, x_k$ are distinct, so are $g_t x_i, g_t x_j, g_t x_k$, so there is a unique $h_t \in G$ such that $$h_t g_t x_i = 0, h_t g_t x_j = \infty, h_t g_t x_k = w_t . $$  Using the fact that the $G$-action is principal on distinct triples and the sequences $h_t g_t x$ and $g_t x$ are approaching the same point at three coordinates, one sees, as in the proof of Proposition~\ref{prop:Chowclass}, that $$ \lim_{t \to \infty} h_t g_t x = \lim_{t \to \infty} g_t x = y . $$  Each $h_tg_t$ fixes $0$ and $\infty$, so the $g_th_t$ range over the $\GG_m < G$ acting by scaling the $\PP^1$.  The sequence $h_tg_t$ cannot converge in $\GG_m$ itself, for then $y = \lim h_tg_t x$ would be in the actual orbit $Gx$, and it cannot converge to $0$ because the $h_tg_t$ move $x_k = 1 \in \PP^1$ to $\infty$ as $t \to \infty$, so the sequence converges to $\infty$.  Thus $h_t g_t$ moves any $z \in \PP^1 \setminus \{ 0 \}$ off to $\infty$ as $t \to \infty$, so $y_l = \lim h_t g_t x_l = \infty$ for every $l \neq i$ since $x_l \neq 0 = x_i$.
\end{proof}

\section{Proof of Lemma \ref{lem:n=4}}\label{section:crossratio}

The proof of Lemma \ref{lem:n=4} is rather straightforward, but the objects that come up during it will be useful in later sections so we go through it carefully now.

\subsection{Classical cross-ratio}

Each $x \in U_4$ has cross-ratio given by \[t(x) :=  \frac{(x_4-x_1)(x_2-x_3)}{(x_2-x_1)(x_4-x_3)} \in \PP^1 \setminus \{ 0,1,\infty \} \] as this formula is well-defined even when $x_i=\infty$ for some $i$.  The function $t(x)$ depends only on the $G$-orbit of $x$, and our convention is that it satisfies $t(0,1,\infty,t)=t$.

\subsection{Multi-homogeneous forms}\label{section:Multi-homogeneous forms}

For fixed $x\in X^4$ the function \begin{equation}\label{eq:form} f_x(z) := (x_4-x_1)(x_2-x_3)(z_2-z_1)(z_4-z_3)  -  (x_2-x_1)(x_4-x_3)(z_4-z_1)(z_2-z_3) \end{equation} is identically zero if and only if $x\in\Delta_\bullet$.  Thus for $x\in X^4\setminus\Delta_\bullet$ we may consider its homogenization \[f_{x} \in \PP\H^0( X^4, \O(1,1,1,1)) = \Hilb(X^4) = \Chow(X^4).\]  If $x\in U_4$ then $f_x$ vanishes on $\Delta_\bullet$ and at points of $U_4$ with the same cross-ratio as $x$, so $f_x$ cuts out the orbit closure $\overline{Gx} \subseteq X^4$ (cf. Lemma \ref{lem:boundary}).  If $x\in \Delta_{\rm big}\setminus\Delta_\bullet$, say $x_1=x_2$, then $f_x$ cuts out the union of the diagonals $z_1=z_2$ and $z_3=z_4$, each of which is a degenerate orbit closure.  

\subsection{Cross-ratio morphism}
We can now define a $G$-invariant morphism \be r : X^4 \setminus \Delta_\bullet & \to & \PP \H^0( X^4, \O(1,1,1,1)) = \PP^{15}\\ x & \mapsto & f_{x}. \ee  In particular, there is an induced map \begin{equation}\label{eq:4incl}M_{0,4} = U_4/G \into \PP\H^0( X^4, \O(1,1,1,1)),\end{equation} and the Hilbert and Chow quotients $X^4\quotient G$ are the closure of the image of this map.  

\subsection{Isomorphism to the Hilbert/Chow quotients}
The morphism $r$ gives a natural way to extend the inclusion (\ref{eq:4incl}) to a morphism \begin{equation}\label{eq:4iso}\ov{M}_{0,4} \rightarrow \PP\H^0( X^4, \O(1,1,1,1)).\end{equation}  Indeed, after choosing a trivialization $M_{0,4}\times\PP^1  \rightarrow M_{0,4}$ of the universal curve, the four sections $p_i$ can be regarded as a map $p : M_{0,4} \to U_4$.  For concreteness, let $p_1,p_2,p_3 : M_{0,4} \to \PP^1$ be the constant maps $0,1,\infty$, respectively.  Then $p : M_{0,4} \to U_4$ extends uniquely to a morphism $\ov{p} : \ov{M}_{0,4} \to X^4 \setminus \Delta_\bullet$ defined by $\ov{p}(t) = (0,1,\infty,t)$.  The composition $r\overline{p}$ is the unique extension of (\ref{eq:4incl}).  We claim that it is an isomorphism onto its image.  Indeed, it is clear from the expression in (\ref{eq:form}) that $f_{(0,1,\infty,t)}$ depends linearly on $t$, so $r\ov{p} : \PP^1 \rightarrow \PP^{15}$ is a non-constant linear map.  Thus $(\ref{eq:4iso})$ is an isomorphism between $\ov{M}_{0,4}$ and the Hilbert and Chow quotients $X^4\quotient G$.

\section{Hilbert polynomial computations}\label{section:Hilbertpolynomial}

We are ready to perform the Hilbert polynomial computations that will be used in the sequel to show that the subschemes $\mathcal{Z}(C,p_1,\ldots,p_n)$ from \S\ref{intsec:Flat family} form a flat family over $\ov{M}_{0,n}$.  Recall that for a closed subscheme $Z \subseteq X^n$, the \emph{Hilbert polynomial} is the function \be p_Z : \ZZ^n & \to & \ZZ \\ (t_1,\dots,t_n) & \mapsto & \chi(Z,\O_{X^n}(t_1,\dots,t_n)|_Z) . \ee  For any such $Z$, $p_Z \in \QQ[t_1,\ldots,t_n]$ is a polynomial function of the $t_i$.  For example, \be p_{X^n}(t_1,\ldots,t_n) & = & \prod_{i=1}^n (t_i+1). \ee

\subsection{Generic orbit closure}

\begin{prop} \label{prop:GenHilbpoly} For any $x \in U_n$, the Hilbert polynomial of $Z := \overline{Gx}$ is: \be p_Z(t_1,\ldots,t_n) & = &\sum_{|I| \le 3}~\prod_{i \in I} t_i . \ee \end{prop}

\begin{proof} If $n=3$ then $Z=X^n$ and the above equality is clear, so assume $n > 3$.  We calculate the Hilbert polynomial of $Z$ by constructing a flat degeneration of $Z$ to a subscheme whose Hilbert polynomial we can compute by induction on $n$.  

Given $x\in U_n$, let $x' := \pi_{[n-1]}(x) \in U_{n-1}$.  Consider the subscheme $W \subseteq X^n \times \PP^1_t$ consisting of those $(y,t) \in X^n \times \PP^1_t$ such that $f_{\pi_I(x',t)}(\pi_I(y)) =0$ for every four-element subset $I \subseteq [n]$.  Here $f_{\pi_I(x',t)}\in \PP\H^0( X^4, \O(1,1,1,1))$ is the multi-homogeneous form introduced in \S\ref{section:Multi-homogeneous forms}.  This $W$ is $G$-invariant for the $G$-action $g \cdot (y,t) = (gy,t)$, but \emph{not} for the diagonal action $g \cdot (y,t) = (gy,gt)$.  In particular, $G$ acts on each fiber of the projection $\pi : W \to \PP^1_t$.  Note that $W$ is the $G$-orbit closure of the section $s(t) :=(x',t,t)$ of $\pi$.  Since $x'$ has at least three distinct coordinates this $G$-orbit is $G \times \PP^1_t$, so $W$ is an integral scheme and $\pi$ is dominant, hence flat.

For $t \in \PP^1 \setminus \{ x_1,\dots,x_{n-1} \}$ the fiber $\pi^{-1}(t) \subseteq X^n \times \{ t \} = X^n$ is the orbit closure of $(x',t) \in X^n$.  Indeed, $(x',t) \in U_n$ so the equations $f_{\pi_I(x',t)}$ cut out the orbit closure of $(x',t)$ in $X^n$.  In particular, the fiber over $x_n$ is $Z$.  For $i \in [n-1]$ we claim that the fiber $Z_i := \pi^{-1}(x_i)$ has two irreducible components $Z_i',Z_i''$ of dimension $3$ such that $Z_i'$ is $\overline{Gx'}$ embedded in the diagonal $\Delta_{i=n} \subseteq X^n$, $Z_i''$ is the diagonal $\Delta_{i,n,[n-1] \setminus i} \cong X^3$, and the components $Z_i', Z_i''$ intersect scheme theoretically in the diagonal $\Delta_{ \{ i,n \}, [n-1] \setminus i} \cong X^2$.

To see this, suppose first that $(y,x_i) \in Z_i$ but $y_i \neq y_n$.  We claim that $y \in Z_i'' = \Delta_{i,n,[n-1] \setminus i}$.  If not, then we would have $y_j \neq y_k$ for some $j,k \in [n-1] \setminus i$, which would violate $f_{\pi_{i,j,k,n}(x',x_i)}(y) = 0$.  Thus $Z_i \subseteq \Delta_{i = n} \cup Z_i''$.  On the other hand, if $(y,x_i) \in Z_i''$, then we have $f_{ \pi_I(x',x_i) }( \pi_I(y)) = 0$ for every four-element subset $I \subseteq [n]$ because either (i) $I$ contains at most one of $i,n$, hence $\pi_I(y) \in \Delta_\bullet \subseteq X^4$ and every form vanishes at $\pi_I(y)$, or (ii) $I = \{ j,k,i,n \}$ and $f_{ \pi_I(x',x_i) }( \pi_I(y)) = 0$ because $y_j=y_k$ and the $i^{\rm th}$ and $n^{\rm th}$ coordinates of $(x',x_i)$ coincide, so $f_{\pi_I(x',x_i)}$ cuts out the union of the hypersurface where the $j,k$ coordinates coincide and the hypersurface where the $i,n$ coordinates coincide.  So $Z_i' \cap Z_i'' = \Delta_{ \{ i,n \}, [n-1] \setminus i}$, scheme-theoretically.

Next, suppose $(y,x_i) \in Z_i \cap (\Delta_{i=n} \times \{ x_i \})$.  We then claim $y \in \Delta_{i=n}(\overline{Gx'})$, hence $Z_i \subseteq Z_i' \cup Z_i''$.  Indeed, as $I$ ranges over four-element subsets of $[n-1]$, we have $f_{\pi_I(x',x_i)} = f_{\pi_I(x')}$, and since $(y,x_i) \in W$, the latter forms vanish at $\pi_I(y)$, so $y \in \Delta_{i=n}(\overline{Gx'}).$   On the other hand, if $(y,x_i) \in Z_i' = \Delta_{i=n}(\overline{Gx'}) \times \{ x_i \}$ then we claim $y \in Z_i$, hence $Z_i = Z_i' \cup Z_i''$.  Indeed, $f_{\pi_I(x',x_i)}(y) = 0$ for every four-element subset $I \subseteq [n]$ because either (i) $I$ contains at most one of $i,n$ and we have the vanishing by definition of $\Delta_{i=n}(\overline{Gx'})$, or (ii) $I = \{ j,k,i,n \}$ and $f_{ \pi_I(x',x_i) }( \pi_I(y)) = 0$ because $y_i=y_n$ and the $i^{\rm th}$ and $n^{\rm th}$ coordinates of $(x',x_i)$ coincide.

For the remainder of the proof set $\Delta := \Delta_{ \{ i,n \}, [n-1] \setminus i}$.  From the exact sequence $$0 \to \O_{Z_i} \to \O_{Z_i'} \times \O_{Z_i''} \to \O_{\Delta} \to 0 $$ and the above flat degeneration, we see that \be p_Z(t_1,\dots,t_n) & = & p_{Z_i}(t_1,\dots,t_n) \\ & = & p_{Z_i'}(t_1,\dots,t_n) + p_{Z_i''}(t_1,\dots,t_n) - p_\Delta(t_1,\dots,t_n) . \ee  By induction, the Hilbert polynomial of a generic orbit closure in $X^{n-1}$ is \be q(t_1,\dots,t_{n-1}) & = & \sum_{J \subseteq [n-1], |J| \leq 3}~\prod_{j \in J} t_j . \ee Since \be \O(t_1,\dots,t_n)|_{\Delta_{i=n}} & = & \O(t_1, \dots,t_{i-1},t_i+t_n,t_{i+1}, \dots, t_{n-1}) \\ \O(t_1,\dots,t_n)|_\Delta & = & \O(t_i+t_n, \sum_{j \neq i,n} t_j ) \\ \O(t_1,\dots,t_n)|_{Z_i''} & = & \O( t_i, t_n, \sum_{j \neq i,n} t_j ) \ee we have \be p_{Z_i'}(t_1,\dots,t_n) & = & q(t_1,\dots,t_{i-1},t_i+t_n,t_{i+1}, \dots, t_{n-1}) \\ & = & \sum_{J \subseteq [n-1]\setminus i, |J| \leq 3}~\prod_{j \in J} t_j + \sum_{J \subseteq [n-1] \setminus i, |J| \leq 2}(t_i+t_n) \prod_{j \in J} t_j \\ p_{Z_i''}(t_1,\dots,t_n) & = & (1+t_i)(1+t_n)(1+\sum_{j \neq i,n} t_j) \\ p_{\Delta}(t_1,\dots,t_n) & = & (1+t_i+t_n)(1+\sum_{j \neq i,n} t_j) .\ee  The rest of the proof is straightforward polynomial arithmetic. \end{proof}

\subsection{Gluing constructions}

For a nodal curve there is an inductive structure to the subscheme $\mathcal{Z}(C,p_1,\ldots,p_n)$, but in order to harness it for the computation of the Hilbert polynomial we first need some preliminary results about gluing.

\begin{lem} \label{lem:ses} Let $Z$ be a reduced scheme, $Z_1, Z_2$ closed subschemes covering $Z$ as topological spaces.  Then the cartesian diagram of schemes $$ \xym{ Z_1 \cap Z_2 \ar[r] \ar[d] & Z_1 \ar[d] \\ Z_2 \ar[r] & Z }$$ is also cocartesian, and we have the following exact sequence of sheaves on $Z$: $$0 \to \O_Z \to \O_{Z_1} \oplus \O_{Z_2} \to \O_{Z_1 \cap Z_2} \to 0. $$ \end{lem}

\begin{proof} This is all local on $Z$, so we can assume $Z=\Spec A$, $Z_i = \Spec A/I_i$, so $Z_1 \cap Z_2 = \Spec A/(I_1+I_2)$.  The sequence of $A$ modules $$0 \to I_1 \cap I_2 \to A \to A/I_1 \oplus A/I_2 \to A/(I_1+I_2) \to 0 $$ is easily seen to be exact, where the last map is the difference of the projections, so it is enough to prove $I_1 \cap I_2 = (0)$.  Since the $Z_i$ cover $Z$, any prime ideal of $A$ contains either $I_1$ or $I_2$ (or both), so any $f \in I_1 \cap I_2$ must be in every prime ideal of $A$, but the intersection of the prime ideals of $A$ is its nilradical, which is zero since $Z$ is reduced. \end{proof}

For $I \subseteq [n]$, recall from \S\ref{section:Indexing sets} that $i_I : X^{I+1} \into X^n$ is the closed embedding defined by repeating the $*$ coordinate at all coordinates in $I^c$.

\begin{lem} \label{lem:gluing}  Suppose $[n] = K + L$ and $Z_K \subseteq X^{K+1}$, $Z_L \subseteq X^{L+1}$ are reduced closed subschemes, each containing $\Delta_\bullet$.  Let $Z := i_K Z_K \cup i_L Z_L$, with the reduced induced structure from $X^n$.  Then: \begin{enumerate} \item \label{1} $Z$ is a reduced closed subscheme of $X^n$ containing $\Delta_\bullet$ and $\Delta_{K,L}$. \item \label{2} The scheme-theoretic intersection $i_K Z_K \cap i_L Z_L$ is $\Delta_{K,L}$. \item \label{3} The Hilbert polynomial of $Z$ is given by \be p_Z(t_1,\dots,t_n) & = & p_{Z_K}(t_K, \sum_{l \in L} t_l) + p_{Z_L}(t_L, \sum_{k \in K} t_k) - (1+\sum_{k \in K} t_k)(1+\sum_{l \in L} t_l)\ee \end{enumerate} where, e.g., $q(t_K,\sum_{l \in L} t_l) \in \QQ[t_1,\dots,t_n]$ means the evaluation of $q \in \QQ[ \{ t_k : k \in K \} , t_*]$ at $t_* = \sum_{l \in L} t_l$.  
\end{lem}

\begin{proof} For $i \in K$ we have $\Delta_i \cap i_K = \Delta_i$ and similarly for $i \in L$, so $Z$ contains $\Delta_\bullet$.  Note that $i_K \cap i_L = \Delta_{K,L}$ so certainly $i_K Z_K \cap i_L Z_L \subseteq \Delta_{K,L}$.  Note also that $i_K \Delta_\bullet = i_L \Delta_\bullet = \Delta_{K,L}$ so $\Delta_\bullet \subseteq Z_K$ implies $\Delta_{K,L} \subseteq i_K Z_K$, and similarly for $L$.  Hence $\Delta_{K,L} \subseteq i_K Z_K \cap i_L Z_L$, so we have proved \eqref{1} and \eqref{2}.  Now \eqref{3} follows from the short exact sequence in Lemma \ref{lem:ses}, the identities \be \O_{X^n}(t_1,\dots,t_n)|_{i_K} & = & \O(t_K, \sum_{l \in L} t_l) \\ \O_{X^n}(t_1,\dots,t_n)|_{i_L} & =& \O(t_L, \sum_{k \in K} t_k) \\ \O_{X^n}(t_1,\dots,t_n)|_{\Delta_{K,L}} & = & \O(\sum_{k \in K} t_k, \sum_{l \in L} t_l), \ee and the fact that $\chi$ is additive on short exact sequences. 
\end{proof}

\subsection{The subschemes $\mathcal{Z}(C,p_1,\ldots,p_n)$}

\begin{prop} \label{prop:Hilbertpolynomial} For any $(C,p_1,\ldots,p_n) \in \ov{M}_{0,n}(k)$, the Hilbert polynomial of the closed subscheme $\mathcal{Z}(C,p_1,\ldots,p_n) \subseteq X^n$ equals the Hilbert polynomial of a generic orbit closure. \end{prop}

\begin{proof} We proceed by induction on the number of components of $C$.  If $C$ is irreducible, then $\mathcal{Z}(C,p_1,\ldots,p_n)$ is a generic orbit closure so the result follows from Proposition \ref{prop:GenHilbpoly}.  Assume $C = C_K \cup_{q} C_L$ is reducible and that the markings split across the two components as indicated according to a partition $[n] = K + L$ with $|K|,|L|\ge 2$.  Let $\mathcal{Z}(C_K,K+q)$ denote the union of component configuration orbit closures for the curve $C_K$ with markings $\{p_i\}_{i\in K}$ and $q$; define $\mathcal{Z}(C_L,L+q)$ analogously.  By the inductive hypothesis the Hilbert polynomials of $\mathcal{Z}(C_K, K+q) \subseteq X^{K+1}$ and $\mathcal{Z}(C_L, L+q) \subseteq X^{L+1}$ are given by the expected formula.  It is clear from the definitions that \be \mathcal{Z}(C,p_1,\ldots,p_n) & = & i_K(\mathcal{Z}(C_K, K+q)) \cup i_L(\mathcal{Z}(C_L, L+q)) \ee and from Lemma~\ref{lem:boundary} that the hypotheses of Lemma~\ref{lem:gluing} are satisfied. Thus, \be p_{\mathcal{Z}(C,p_1,\ldots,p_n)} & = & p_{\mathcal{Z}(C_K, K+q)} + p_{\mathcal{Z}(C_L, L+q)} - p_{\Delta_{K,L}} \\ & = & \sum_{I \subseteq K, |I| \leq 3}~\prod_{i \in I} t_i + \left ( \sum_{l \in L} t_l \right ) \sum_{I \subseteq K, |I| \leq 2}~\prod_{i \in I} t_i \\ & & + \sum_{I \subseteq L, |I| \leq 3}~\prod_{i \in I} t_i + \left ( \sum_{k \in K} t_k \right ) \sum_{I \subseteq L, |I| \leq 2}~\prod_{i \in I} t_i \\ & & - (1 + \sum_{k \in K} t_k)(1+ \sum_{l \in L} t_l ) \\ & = & \sum_{I \subseteq [n], |I| \leq 3}~\prod_{i \in I} t_i . \ee as claimed. \end{proof} 

\section{Proof of Theorem \ref{thm:Zflat}}\label{section:Proof of Zflat}

By the results of the previous section (Proposition \ref{prop:Hilbertpolynomial}) the subschemes $\mathcal{Z}(C,p_1,\ldots,p_n)$ all have the same Hilbert polynomial, so to show that they form a flat family we just need to show that they fit together into an algebraic family.  We will do this by constructing the family over $M_{0,n}$ first and then showing that it extends uniquely to a flat family over $\ov{M}_{0,n}$ which must in fact be the desired family.

\subsection{The family over the interior}

For each $I = \{ i,j,k \} \subseteq [n]$ there is a unique identification of the universal curve over $M_{0,n}$ with $M_{0,n} \times X$ so that $p_i,p_j,p_k$ become the constant sections $0,1,\infty$.  After making this identification, the sections $(p_1,\ldots,p_n)$ can be viewed as a map \[p_I : M_{0,n} \rightarrow U_n.\]  Another three-element subset of the markings yields three disjoint sections which necessarily differ from the constant sections $0,1,\infty$ by some $g \in \PGL_2(M_{0,n})$.  In particular, the $G$-orbit $Z^\circ \subseteq M_{0,n} \times U_n$ of the graph of $p_I$ is independent of the choice of $I$.  Let \[Z\subseteq M_{0,n} \times X^n\] be the closure of $Z^\circ$, with reduced induced structure.  

\begin{lem}\label{lem:openflat} We have $Z = Z^\circ \cup (M_{0,n} \times \Delta_\bullet)$, and it is flat over $M_{0,n}$ with fiber $\mathcal{Z}(C,p_1,\ldots,p_n)$ over $(C,p_1,\ldots,p_n) \in M_{0,n}$. \end{lem}

\begin{proof} Clearly $Z$ contains each $M_{0,n} \times \Delta_i$ because for every point $x \in M_{0,n}$ the closure of $Z^\circ_x$ in the fiber over $x$ contains $\Delta_i$ by Lemma~\ref{lem:boundary}.  To see that no other points are in $Z$ we first consider $n=4$.  Let  $Z_4 \subseteq M_{0,4} \times X^4$ be the same definition as $Z$ but in the case $n=4$.  Identify $M_{0,4}$ with the space of multi-homogeneous forms as in \S\ref{section:crossratio}: $$ M_{0,4} = \PP^1 \setminus \{ 0,1,\infty \} \subseteq \PP \H^0( X^4, \O(1,1,1,1) ). $$ Then $Z_4 = \{ (f,x) : f(x)=0 \}$ and this has the claimed description by the discussion in \S\ref{section:crossratio} of the zero loci of these $f$.  For arbitrary $n$, note that if $x \in X^n$ is not in any of the $\Delta_i$ then there is $I \subseteq [n]$ of size four such that $\pi_I(x)$ is not in any of the $\Delta_i \subseteq X^I$.  The result then follows from the fact that the diagram $$ \xym@C+20pt{ M_{0,n} \times X^n \ar[r]^-{s_I \times \pi_I} \ar[d] & M_{0,4} \times X^4 \ar[d] \\ M_{0,n} \ar[r]^-{s_I} & M_{0,4} }$$ is commutative and $G$-equivariant and the top arrow takes $Z$ onto $Z_4$.  This proves the claimed description of $Z$; flatness follows since the $Z(C,p_1,\ldots,p_n)$ all have the same Hilbert polynomial by Proposition~\ref{prop:GenHilbpoly}. \end{proof}

It follows from this lemma that there is an induced morphism $Z : M_{0,n} \rightarrow \Hilb(X^n)$ sending a configuration to its orbit closure.  This is the inclusion from the definition of the Hilbert quotient---see (\ref{eq:HilbIncl}) in \S\ref{intsec:Hilbert quotient}.  Our goal then is to show that $Z$ extends to a morphism $\ov{Z} : \ov{M}_{0,n} \rightarrow \Hilb(X^n)$ such that the image of $(C,p_1,\ldots,p_n)$ is the subscheme $\mathcal{Z}(C,p_1,\ldots,p_n)$ for all curves, not just smooth ones.

\subsection{Interlude on extensions and DVRs}

We include here a general result about DVRs and extending morphisms that will be subsequently applied to our specific situation.

\begin{definition}
Let $(A,\m)$ be a DVR with residue field $k$ and fraction field $K$, and let $Y$ be a proper scheme.  By the valuative criterion, any map $g : \Spec K \to Y$ extends to a map $\ov{g} : \Spec A \to Y$.  We write $\lim g$ for the point $\ov{g}(\m) \in Y$.
\end{definition}

\begin{thm}\label{thm:mapextension} 
Suppose $X_1,X_2$ are proper schemes over a noetherian scheme $S$ with $X_1$ normal.  Let $U \subseteq X_1$ be an open dense set and $f : U \to X_2$ an $S$-morphism.  Then $f$ extends to an $S$-morphism $\ov{f} : X_1 \to X_2$ if, and only if, for any DVR $(A,\m)$ as above and any morphism $g : \Spec K \to U$, the point $\lim fg$ of $X_2$ is uniquely determined by the point $\lim g$ of $X_1$.  
\end{thm}

\begin{proof} It follows from continuity that an extension implies the statement about DVRs, so we consider the converse.  Let $\Gamma_f \subseteq U \times_S X_2$ be the graph of $f$, so the projection $\pi^\circ_1 : \Gamma_f \to U$ is an isomorphism.  Let $\ov{\Gamma}_f \subseteq X_1 \times_S X_2$ be the closure of $\Gamma_f$, and let $\pi_i : \ov{\Gamma}_f \to X_i$ be the projections.  To complete the proof it is enough to show that $\pi_1$ is an isomorphism, since then $\ov{f} := \pi_2 \pi_1^{-1}$ will be the desired extension.  

Note that $\ov{\Gamma_f}$ contains $\Gamma_f$ as an open dense subset so $\pi_1$ is birational.  The $X_i$ are proper over $S$, so $\pi_1$ is also proper, and $X_1$ is normal, so to conclude that $\pi_1$ is an isomorphism it is enough to show that it is a finite morphism.  If $\pi_1$ were not finite, then by Lemma~\ref{lem:finite}, below, we could find a closed point $x \in X_1$ with $\pi_1^{-1}(x)$ infinite.  By Lemma~\ref{lem:points}, below, we could then find points $y,y' \in \pi_1^{-1}(x)$ with $\pi_2(y) \neq \pi_2(y')$.  Since $\Gamma_f \subseteq \ov{\Gamma}_f$ is dense, by \cite[II.7.1.9]{EGA} we can find DVRs $(A,\m,k,K)$, $(A',\m',k',K')$, and maps $g : \Spec K \to \Gamma_f$, $g' : \Spec K' \to \Gamma_f$ with $\lim g = y$, $\lim g' = y'$.  Now $\lim \pi^\circ_1 g = \pi_1(g(\m)) = \pi_1(y)=x$ and similarly $\lim \pi^\circ_1 g' = x$.  But $\lim f \pi^\circ_1 g = \lim \pi_2 g = \pi_2(y)$, and similarly $\lim f \pi^\circ_1 g' = \pi_2(y')$, which contradicts the hypothesis. \end{proof}

\begin{lem} \label{lem:finite} Let $X_1$ be a noetherian scheme and $f : X_1 \to X_2$ a proper morphism.  Then $f$ is finite if and only if $f^{-1}(x)$ is finite over $k(x)$ for every closed point $x \in X_2$. \end{lem}

\begin{proof} One direction is clear.  To see the converse, note that a proper quasi-finite morphism is finite, so it suffices to show that $f^{-1}(x)$ is finite over $k(x)$ for every $x \in X_2$.  Since $f$ is finite type, $f^{-1}(x)$ is finite over $k(x)$ if and only if $f^{-1}(x)$ is zero-dimensional.  By upper semi-continuity of fiber dimension \cite[IV.13.1.5]{EGA}, the set $T$ of $x \in X_2$ where $\dim f^{-1}(x) > 0$ is closed.  But $X_2$ is a Zariski space, so if $T$ were non-empty then it would contain a closed point. \end{proof}

\begin{lem} \label{lem:points} Let $f_1 : X_1 \to S$ be a proper morphism of noetherian schemes, and let $f_2 : X_2 \to S$ be a morphism of locally finite type.  Let $Y \subseteq X_1 \times_S X_2$ be a closed subset and $\pi_i : Y \to X_i$ the projections.  Then for any closed point $x \in X_1$, if $\pi_1^{-1}(x)$ is not finite then there are points $y,y' \in \pi_1^{-1}(x) $ with $\pi_2(y) \neq \pi_2(y')$. \end{lem}

\begin{proof} Suppose, to the contrary, that $\pi_2( \pi_1^{-1}(x) )$ consists of a single point $y$.  Since $f_1$ is proper and $x$ is closed, $\pi_2(\pi_1^{-1}(x))$ is a closed subscheme of $X_2$.  Set $s = f_2(y) = f_1(x) \in S$.  Since $f_2$ is locally of finite type, so is its restriction to the closed subscheme $\pi_2(\pi_1^{-1}(x))$ whose underlying topological space is the single point $y$. Thus the residue field $k(y)$ must be a finite field extension of $k(s)$.  By \cite[I.3.4.9]{EGA}, the set $T$ of points of $X_1 \times_S X_2$ mapping to $x$ via the first projection and $y$ via the second is in bijective correspondence with the set of points of $\Spec k(x) \otimes_{k(s)} k(y)$, so $T$ is finite since $k(s) \into k(y)$ is a finite field extension.  But clearly $\pi_1^{-1}(x)$ is contained in $T$, so this contradicts the assumption that $\pi_1^{-1}(x)$ is infinite. \end{proof}

\subsection{Extending the flat family}

We would like to apply Theorem \ref{thm:mapextension} to the morphism $Z : M_{0,n} \rightarrow \Hilb(X^n)$, to obtain an extension $\ov{Z} : \ov{M}_{0,n} \rightarrow \Hilb(X^n)$.  To do this, we need to check that given a DVR $(A,\m,k,K)$ and a $K$-point $(C_K,p_1,\ldots,p_n)\in M_{0,n}(K)$ the flat limit of the subscheme $\mathcal{Z}(C_K,p_1,\ldots,p_n)\in\Hilb(X^n)(K)$ is the subscheme $\mathcal{Z}(C_k,p_1,\ldots,p_n)\in\Hilb(X^n)(k)$ constructed from the curve $(C_k,p_1,\ldots,p_n)$ which is the flat limit of $(C_K,p_1,\ldots,p_n)$.  In fact, not only will this abstractly yield an extension by Theorem \ref{thm:mapextension} but it explicitly shows that the extension is given by $(C,p_1,\ldots,p_n) \mapsto \mathcal{Z}(C,p_1,\ldots,p_n)$ and hence that the $\mathcal{Z}(C,p_1,\ldots,p_n)$ form a flat family over $\ov{M}_{0,n}$, thereby completing the proof of Theorem \ref{thm:Zflat}.

It follows from \cite[Lemma 3.2]{Gi} that \[\mathcal{Z}(C_k,p_1,\ldots,p_n)\subseteq 
\lim \mathcal{Z}(C_K,p_1,\ldots,p_n) \in \Hilb(X^n)(k).\]  But the definition of flat limit implies that $\lim \mathcal{Z}(C_K,p_1,\ldots,p_n)$ has the same Hilbert polynomial as $\mathcal{Z}(C_K,p_1,\ldots,p_n)$, and by Proposition \ref{prop:Hilbertpolynomial} we know that $\mathcal{Z}(C_K,p_1,\ldots,p_n)$ has the same Hilbert polynomial as $\mathcal{Z}(C_k,p_1,\ldots,p_n)$, so this inclusion is in fact an equality.
  
\section{Proof of Proposition \ref{prop:det}}\label{Proof of det} 

Fix $I\subseteq [n]$ of cardinality four and consider the projection \[\pi := (\pi_I,\text{id}) : X^n \times (X^n\HilbQ G) \rightarrow X^4 \times (X^n\HilbQ G).\]  Let $Y\subseteq X^n \times (X^n\HilbQ G)$ be the restriction of the universal family over $\Hilb(X^n)$ to $X^n\HilbQ G \subseteq \Hilb(X^n)$.  Applying the derived direct image functor  $R\pi_*$ to the restriction map $\mathcal{O}_{X^n \times (X^n\HilbQ G)} \rightarrow \mathcal{O}_Y$ yields \[R\pi_*\mathcal{O}_{X^n \times (X^n\HilbQ G)} \rightarrow R\pi_*\mathcal{O}_Y.\]  The fibers of $\pi$ are products of $\PP^1$ so $R\pi_*\mathcal{O}_{X^n \times (X^n\HilbQ G)}=\mathcal{O}_{X^4 \times (X^n\HilbQ G)}$.  We claim that $R\pi_*\mathcal{O}_Y$ is a perfect complex, i.e., locally isomorphic in the derived category to a bounded complex of vector bundles of finite rank.  Indeed, $Y$ is flat over $X^n\HilbQ G$ so by \cite[III.3.6]{SGA} $\mathcal{O}_Y$ has finite tor-dimension as a sheaf on $X^n \times (X^n\HilbQ G)$.  The claim now follows from the fact that $\pi$ is smooth and everything is proper. 

We can thus apply the determinant functor of Knudsen-Mumford \cite[Theorem 2]{KM}: \[\mathcal{O}_{X^4 \times (X^n\HilbQ G)} \rightarrow \det (R\pi_*\mathcal{O}_Y)\] yielding a section $t$ of the line bundle $\det (R\pi_*\mathcal{O}_Y)$ on $X^4 \times (X^n\HilbQ G)$.  For each point $x\in X^n\HilbQ G$ the fiber $Y_x\subseteq X^n$ of $Y$ over $x$ is a 3-dimensional closed subscheme of the form $Y_x = \mathcal{Z}(C,p_1,\ldots,p_n)$ for some $(C,p_1,\ldots,p_n)\in\ov{M}_{0,n}$.  By definition the determinant functor commutes with base change so $t_x$, the restriction of $t$ to the fiber over $x$, is given by the determinant of $\mathcal{O}_{X^4} \rightarrow R\pi_*\mathcal{O}_{Y_x}$.  If $C$ is smooth then it follows from Lemma \ref{lem:boundary} that $\pi(Y_x)$ is the divisor $\mathcal{Z}(s_I(C,p_1,\ldots,p_n))\subseteq X^4$, so $t_x$ is the section $f_{s_I(C,p_1,\ldots,p_n)}$ from \S\ref{section:Multi-homogeneous forms} cutting out this divisor.  Thus we have a map \be U_n/G & \to & X^4\HilbQ G \subseteq \PP\H^0(X^4, \O(1,1,1,1))\\ x & \mapsto & t_x = f_{s_I(C,p_1,\ldots,p_n)}. \ee  This map extends to $X^n\HilbQ G$ because for any fiber $Y_x = \mathcal{Z}(C,p_1,\ldots,p_n)$ the image $\pi(Y_x)\subseteq X^4$ is a divisor with the same Chow class as a generic orbit closure (cf. Proposition \ref{prop:Chowclass}) so $0\ne t_x \in \H^0(X^4, \O(1,1,1,1))$.  Thus for each  $I$ of cardinality four we have a diagram \[\xymatrix{\ov{M}_{0,n} \ar[r] \ar[d] & X^n\HilbQ G \ar[d] \\ \ov{M}_{0,4} \ar[r] & X^4\HilbQ G}\] which must be commutative since it commutes on dense open subsets and everything is separated.  Taking the product over all $I$ of cardinality four completes the proof.

\appendix
\section{Categorical framework}

We explain here two ways in which the structure of the results/proofs in this paper formally reduce to the case $n=4$.  As before, we set $X := \PP^1$.

\subsection{Natural transformations}

Let $\mathcal{S}$ be the category of finite sets of cardinality at least three, with injections as morphisms, and let $\mathcal{S}ch$ be the category of schemes.  The stabilization maps furnish the association $I \mapsto \ov{M}_{0,I}$ with the structure of a contravariant functor, denoted $\ov{M}_{0,\bullet}$, since if $I\subseteq J \subseteq K$ then the following diagram is commutative: \[\xym{ \ov{M}_{0,K} \ar[r]^{s_J} \ar[rd]_{s_I} & \ov{M}_{0,J} \ar[d]^{s_I} \\ & \ov{M}_{0,I} }\]  This affords us a slight generalization of the argument used in the introduction.

\begin{lem}
Given a functor $F : \mathcal{S}^{\text{op}} \rightarrow \mathcal{S}ch$ and a natural transformation $\eta : \ov{M}_{0,\bullet} \rightarrow F$, for $\eta$ to be an isomorphism it is necessary and sufficient that $\eta_I : \ov{M}_{0,I} \rightarrow F(I)$ be surjective for all $I$ and an isomorphism for $|I|=4$. 
\end{lem}

\begin{proof}
One direction is automatic.  For the other direction, we must show that given any finite set $I$ the surjection $\eta_I : \ov{M}_{0,I} \rightarrow F(I)$ is an isomorphism.  By definition, for each $J\subseteq I$ of cardinality four there is a commutative diagram \[\xymatrix{\ov{M}_{0,I} \ar[r] \ar[d] & F(I) \ar[d] \\ \ov{M}_{0,J} \ar[r] & F(J)}\]  Taking the product of the bottom row over all such $J$ yields \[\xymatrix{\ov{M}_{0,I} \ar[r] \ar[d] & F(I) \ar[d] \\ \prod_{|I|=4}\ov{M}_{0,J} \ar[r] & \prod_{|J|=4}F(J)}\] By assumption the bottom row is an isomorphism, and by Theorem \ref{thm:prodstab} the left arrow is an isomorphism onto its image, so the top arrow must be as well. 
\end{proof} 

This implies that $\ov{M}_{0,n} \rightarrow X^n\ChowQ G$ is an isomorphism, since functoriality for these Chow varieties follows from \cite[Theorem 6.8]{Ko} and the $n=4$ hypothesis is guaranteed by Lemma \ref{lem:n=4}.  For the Hilbert quotients we do not have functoriality, as Proposition \ref{prop:det} only gives maps corresponding to inclusions $I\subseteq J$ with $|I|=4$.  This is all that is needed to deduce that $\ov{M}_{0,n} \rightarrow X^n\HilbQ G$ is an isomorphism, and indeed that was our approach in the introduction.  However, in the next section we explain how there is a natural locus in $\Hilb(X^n$), containing $X^n\HilbQ G$, such that these loci form a functor $\mathcal{S}^{\text{op}} \rightarrow \mathcal{S}ch$.

\subsection{Boundary compatibility}

The preceding discussion illustrates how certain questions about maps from $\ov{M}_{0,n}$ can be reduced to the case $n=4$ by applying Theorem \ref{thm:prodstab}.  Our proof of Theorem \ref{thm:prodstab} used a certain compatibility between the stabilization maps and the boundary inclusions.  This same compatibility can be used to show directly that the maps from $\ov{M}_{0,n}$ to the Hilbert and Chow quotients are isomorphisms, without appealing to Theorem \ref{thm:prodstab}.  In fact, this approach yields a unified proof of Theorems \ref{thm:prodstab}, \ref{thm:Ziso}, and \ref{thm:ZChiso}.  The key is to formalize the compatibility between stabilization and boundary.

\subsubsection{Cyclic operads} The combinatorial data of the boundary inclusions has been abstracted by Getzler and Kapranov to the notion of a cyclic operad \cite{GK}.  Recall that in a symmetric monoidal category $\mathcal{C}$, this is a functor $F: \mathcal{S}' \rightarrow \mathcal{C}$, where $\mathcal{S}'$ is the category of finite sets with bijections as morphisms, together with a composition morphism \[i\circ j : F(I)\otimes F(J) \rightarrow F(I\sqcup J \setminus\{i,j\})\] for each pair of sets $I,J$ and elements $i\in I, j\in J$.  This data is required to satisfy certain natural axioms---see \cite[\S2.1]{GS} for details.  The quintessential example of a cyclic operad is the functor $\ov{M}_{0,\bullet}$, where the composition morphisms are the boundary inclusions \[\ov{M}_{0,I+1}\times\ov{M}_{0,J+1} \into \ov{M}_{0,I+J}.\]  However, by working with the category $\mathcal{S}'$ we lose the data of the stabilization morphisms.  It is convenient to work instead with the category $\mathcal{S}$ with inclusions as morphisms.

\begin{definition}
A \emph{procyclic operad} is a functor $F: \mathcal{S}^{\text{op}} \rightarrow \mathcal{C}$ satisfying the axioms of a cyclic operad in addition to the analogs of the diagrams in \S\ref{section:Stabilization morphisms}.  For $I'\subseteq I$, we call $F(I) \rightarrow F(I')$ the \emph{projection}.  A procyclic operad in $\mathcal{S}ch$ is \emph{closed} if the composition morphisms are proper.  In this case we define the \emph{boundary} of $F(I)$ to be the union of the images of the composition morphisms and the \emph{interior} to be its complement.  A morphism of procyclic operads is a natural transformation $F \to G$ of functors $\mathcal{S}^{\rm op} \to \mathcal{C}$ commuting with the composition morphisms: $$ \xym@C+40pt{ F(I+1) \otimes F(J+1) \ar[d] \ar[r] & G(I+1) \otimes G(J+1)\ar[d] \\ F(I+J) \ar[r] & G(I+J) }$$.\end{definition}

\begin{remark} A procyclic operad is a cyclic operad endowed with the data of projection morphisms.  It is also related to the notion of a $\Gamma$-operad, a structure introduced by Behrend and Manin to explain the appearance of various types of operads in moduli theory.  See \cite[\S5]{BM}.  A $\Gamma$-operad in a category $\mathcal{C}$ is a monoidal functor $\Gamma \rightarrow \mathcal{C}$, where $\Gamma$ is the category of stable trees, suitably defined. A procyclic operad $F : S^\text{op} \rightarrow \mathcal{C}$ determines a $\Gamma$-operad $\Gamma \rightarrow \mathcal{C}$ by sending each tree $\tau$ to the product, over all vertices $v\in\tau$, of the objects $F(I_v)$, where $I_v$ denotes the set of edges in $\tau$ incident to $v$.  The definition for morphisms is more subtle, but it turns out that the axioms for a procyclic operad are enough to produce a well-defined $\Gamma$-operad.
\end{remark}

Clearly $\ov{M}_{0,\bullet}$ is a closed procyclic operad with stabilization maps as projections and boundary inclusions as compositions.  The boundary and interior defined abstractly here coincide with their usual topological meaning.  As we explain below, the Chow varieties $\Chow(X^\bullet)$ and a suitable locus inside the Hilbert schemes $\Hilb(X^\bullet)$ each form a closed procyclic operad as well.  Another example is provided by the functor \[P(I) := \prod_{I' \subseteq I, |I'|=4} \ov{M}_{0,4}.\]  For $K \subseteq I$ the projection map \[\prod_{I' \subseteq I, |I'|=4} \ov{M}_{0,4} \rightarrow \prod_{K' \subseteq K, |K'|=4} \ov{M}_{0,4} \] is defined in the obvious way.  The composition morphisms are defined as follows.  The image of $x\in P(I+1)\times P(J+1)$ in $P(I+J)=\prod_{K\subseteq I+J,|K|=4}\ov{M}_{0,4}$ has for its $K$-component $x_K$ if $K\subseteq I$ or $K\subseteq J$, $x_{K+1}$ if $|K\cap I|=3$ or $|K\cap J|=3$, and if $|K\cap I|=|K\cap J| =2$ then it is the point in $\ov{M}_{0,4}$ corresponding to a reducible curve with markings $K\cap I$ on one component and $K\cap J$ on the other.  It is straightforward to check that these definitions satisfy the axioms of a closed procyclic operad, and that the product of stabilization maps over four-element subsets yields a morphism $\ov{M}_{0,\bullet} \rightarrow P$.

\subsubsection{Hilbert and Chow operads}

It is not quite true that the Hilbert schemes $\Hilb(X^\bullet)$ form a procyclic operad, since for instance there are no projection morphisms.  However, we can define a locus inside these Hilbert schemes which has all the desired properties.  Let $\Hilb' X^I \subseteq \Hilb X^I$ be the locally closed subscheme parameterizing reduced closed subschemes of $X^I$ containing $\Delta_\bullet$.  It follows from Lemma~\ref{lem:gluing} that for a partition $I = K+L$ with $|K|,|L| \geq 2$ there is a gluing morphism \be \delta_{K,L} : \Hilb' X^{K +1} \times \Hilb' X^{L+1} & \to & \Hilb' X^I \\ (W_1,W_2) & \mapsto & i_K( W_1) \cup i_L (W_2) . \ee This map is a closed embedding onto the locus in $\Hilb' X^I$ parameterizing closed subschemes of $X^I$ contained in $i_K \cup i_L$ and containing $\Delta_{K,L}$.  Let $W \subseteq (\Hilb' X^I) \times X^I$ denote the universal closed subscheme and $\Hilb^{\circ}X^I$ the maximal open subset over which the closed subscheme $(\Id \times \pi_J)(W) \subseteq (\Hilb' X^I) \times X^J$ is flat for every $J\subseteq I$ of cardinality at least three.  There is an induced morphism $\Hilb^{\circ} X^I \rightarrow \Hilb^{\circ} X^J$, and with a little work one can show that this furnishes $\Hilb^{\circ} X^\bullet$ with the structure of a closed procyclic operad, where the composition morphisms are defined as the gluing maps $\delta_{K,L}$.  Moreover, one can show that the map $\ov{Z} : \ov{M}_{0,I} \to \Hilb X^I$ from Theorem \ref{thm:Zflat} factors through $\Hilb^{\circ} X^I$ and in fact yields a morphism $\ov{M}_{0,\bullet} \rightarrow \Hilb^{\circ} X^\bullet$ of procyclic operads.

For $\Chow(X^I)$ the situation is simpler: the maps $i_J : X^{J+1} \into X^I$ and $\pi_J : X^I \to X^J$ are both proper so by \cite[Theorem 6.8]{Ko} they induce morphisms of the corresponding Chow varieties.  By defining the projections to be $\pi_{J*}$ and the composition morphisms as \[ \delta_{K,L} := i_{K*}+i_{L*} : \Chow_\beta X^{K+1} \times \Chow_\beta X^{L+1} \to \Chow_\beta X^I\] we have that $\Chow(X^\bullet)$ is a closed procyclic operad and $[\ov{Z}] : \ov{M}_{0,\bullet} \rightarrow \Chow(X^\bullet)$ is a morphism of procyclic operads, where as usual these Chow varieties parameterize effective cycles with class $\beta$ from Proposition \ref{prop:Chowclass}.

\subsubsection{Morphisms from $\ov{M}_{0,\bullet}$}

We now come to our main result about procyclic operads, the proof of which is only a slight modification of our proof of Theorem \ref{thm:prodstab} given in \S\ref{section:Separating}.

\begin{thm} \label{thm:closedsystems} Let $\phi : \ov{M}_{0,\bullet} \to F$ be a morphism of closed procyclic operads such that (i) $\phi_I : \ov{M}_{0,I} \rightarrow F(I)$ is an isomorphism onto its image when $|I| = 4$, and (ii) for every $I$, $\phi_I$ embeds $M_{0,I}$ into the interior of F(I).  Then every $\phi_I$ is an isomorphism onto its image. \end{thm}

\begin{proof} We proceed by induction on the cardinality of $I$, where the base case is  assumed in (i).  By (ii) we have that $\phi_I$ separates any two points in the interior, so assume $x \in D_{K,L}$, $y \neq x$.  If $y \in D_{K,L}$, then $\phi_I(x)\ne\phi_I(y)$ by induction using the commutativity requirement in the definition of a morphism of procyclic operads.  So assume $y \notin D_{K,L}$.  Then, as we saw in the proof of Theorem~\ref{thm:prodstab}, we can find $k,k' \in K, l,l' \in L$ such that $s_J(y) \neq s_J(x)$ for $J = \{ k,k',l,l' \}$.  We conclude that $\phi_I(x) \neq \phi_I(y)$ from (i) and the commutativity of $$ \xym{ \ov{M}_{0,I} \ar[r] \ar[d] & \ov{M}_{0,J} \ar[d] \\ F(I) \ar[r] & F(J) } $$  

Next we show that $\phi_I$ separates tangent vectors at a point $x \in \ov{M}_{0,I}$.  By (i) and induction we reduce to showing that, for $x \in D_{K,L}$, the differential of $\phi_I$ at $x$ does not kill the normal of $D_{K,L}$ in $\ov{M}_{0,I}$ at $x$.  As in the proof of Theorem~\ref{thm:prodstab}, we can find a four-element subset $J$ as in the above paragraph so that the differential of $s_J$ at $x$ does not kill this normal.  The result we want follows from commutativity of the diagram \small $$ \xym@R+15pt{ \ov{M}_{0,K+1} \times \ov{M}_{0,L+1} \ar[rd]^{\phi_{K+1} \times \phi_{L+1}} \ar[dd] \ar[rr]^{s_{k,k',*} \times s_{l,l',*}} & & \ov{M}_{0,k,k',*} \times \ov{M}_{0,l,l',*} \ar[dd] \ar[rd]^{\phi_{k,k',*} \times \phi_{l,l',*}}  \\ & F(K+1) \times F(L+1) \ar[dd] \ar[rr] & & F(k,k',*) \times F(l,l',*) \ar[dd] \\ \ov{M}_{0,I} \ar[rr] \ar[rd]^{\phi_I} & & \ov{M}_{0,J} \ar[rd]^{\phi_J} \\ & F(I) \ar[rr] & & F(J)}  $$ \normalsize and the fact that $\phi_{J}$ is an isomorphism onto its image by (i), so its differential does not kill the normal of $\ov{M}_{0,k,k',*} \times \ov{M}_{0,l,l',*}$ in $\ov{M}_{0,J}$ at $s_J(x)$.  \end{proof}

In addition to the operad $P$ discussed earlier, this theorem applies to the Hilbert and Chow operads defined above, yielding a proof of the isomorphisms $\ov{M}_{0,n} \rightarrow X^n\HilbQ G$ and $\ov{M}_{0,n} \rightarrow X^n\ChowQ G$ that is slightly different than the one described in the introduction.


\begin{thebibliography}{EGA}

\bibitem{BM} Behrend K, Manin Y. Stacks of stable maps and Gromov-Witten invariants. Duke Math J 1996; 85(1): 1--60.
\bibitem{EGA} Dieudonn\'e J, Grothendieck A. \'El\'ements de g\'eom\'etrie alg\'ebrique. Publications Math\'ematiques de l'IH\'ES, 1960.
\bibitem{GK}Getzler E, Kapranov M. Cyclic operads and cyclic homology. In: Yau ST, editor. Geometry, Topology and Physics for R. Bott. Cambridge, MA, USA: International Press, 1995, pp. 167--201.
\bibitem{Gi}Giansiracusa N. Conformal blocks and rational normal curves.  J Algebraic Geom 2013; 22: 773--793.
\bibitem{GS}Giansiracusa JH, Salvatore P. Cyclic operad formality for compactified moduli spaces of genus zero surfaces. T Am Math Soc 2012; 364(11): 5881--5911.
\bibitem{SGA} Grothendieck A, Illusie L, Pierre B. S\'eminaire de G\'eom\'etrie Alg\'ebrique du Bois Marie: Th\'eorie des intersections et th\'eo\`eme de Riemann-Roch. Lecture notes in mathematics 225.  Berlin, Germany: Springer-Verlag, 1971.
\bibitem{K1} Kapranov M. Chow quotients of Grassmannians, I. In: Gelfand Seminar, Adv Soviet Math 16(2). Providence, RI, USA: Amer Math Soc, 1993, pp. 29--110.
\bibitem{K2} Kapranov M. Veronese curves and Grothendieck-Knudsen moduli 
space $\ov{M}_{0,n}$. J Algebraic Geom 1993; 2(2): 239--262.
\bibitem{Kn} Knudsen F. The projectivity of the moduli space of stable curves, II: the stacks $\ov{M}_{g,n}$. Math Scand 1983; 52(2): 161--199.
\bibitem{KM} Knudsen F, Mumford D. The projectivity of the moduli space of stable curves, I: preliminaries on ``det'' and ``Div''.  Math Scand 1976; 39: 19--55.
\bibitem{Ko}Koll\'ar J. Rational curves on algebraic varieties. Secaucus, NJ, USA: Springer, 1996.
\end{thebibliography}
\end{document}